\documentclass[11pt,a4paper]{amsart}
\usepackage[margin=1in]{geometry}

\usepackage{amsmath}
\usepackage{amssymb}

\usepackage[only,llbracket,rrbracket]{stmaryrd}
\usepackage{indentfirst}

\usepackage{tikz-cd}
\usepackage[english]{babel}
\usepackage{framed}
\usepackage{xcolor}
\usepackage[pagebackref]{hyperref}
\hypersetup{
    colorlinks=true,
    linkcolor=blue,
    filecolor=magenta,   
    citecolor=purple,
    urlcolor=blue,
    pdftitle={Perv-Hodge},
    pdfpagemode=FullScreen,
    }

\theoremstyle{plain}
\newtheorem{thm}{Theorem}[section]
\newtheorem{prop}[thm]{Proposition}
\newtheorem{lem}[thm]{Lemma}
\newtheorem{cor}[thm]{Corollary}

\newtheorem{defn}[thm]{Definition}

\newtheorem{ques}{Question}[section]

\newtheorem{eg}{Example}[section]

\newtheorem*{rem}{Remark}

\theoremstyle{remark}

\newcommand{\gr}{\operatorname{gr}}
\newcommand{\dr}{\operatorname{DR}}
\newcommand{\cd}{\operatorname{codim}}
\newcommand{\supp}{\operatorname{Supp}}
\newcommand{\coh}{\operatorname{coh}}
\newcommand{\cB}{\mathcal{B}}
\newcommand{\cC}{\mathcal{C}}
\newcommand{\cD}{\mathcal{D}}
\newcommand{\cE}{\mathcal{E}}
\newcommand{\cF}{\mathcal{F}}
\newcommand{\cG}{\mathcal{G}}
\newcommand{\cH}{\mathcal{H}}
\newcommand{\cI}{\mathcal{I}}

\newcommand{\cM}{\mathcal{M}}
\newcommand{\cN}{\mathcal{N}}
\newcommand{\cO}{\mathcal{O}}
\newcommand{\cP}{\mathcal{P}}
\newcommand{\cS}{\mathcal{S}}
\newcommand{\cT}{\mathcal{T}}

\newcommand{\CC}{\mathbb{C}}
\newcommand{\DD}{\mathbb{D}}
\newcommand{\NN}{\mathbb{N}}
\newcommand{\PP}{\mathbb{P}}
\newcommand{\QQ}{\mathbb{Q}}
\newcommand{\ZZ}{\mathbb{Z}}

\newcommand{\fp}{\mathfrak{p}}
\newcommand{\fX}{\mathfrak{X}}

\newcommand{\ic}{\operatorname{IC}}
\newcommand{\cIC}{\mathcal{IC}}
\newcommand{\HM}{\operatorname{HM}}

\newcommand{\Perv}{\operatorname{Perv}}
\newcommand{\pervD}{{^p}D}
\newcommand{\pervH}{{^p}\cH}
\newcommand{\mervD}{{^m}D}

\newcommand{\Dbcoh}{D^b_{coh}}
\newcommand{\pnum}{{^\fp}Ih}
\newcommand{\Dtcoh}[1]{D_{\mathit{coh}}^{#1}}

\newcommand{\mDtcoh}[2]{ {^{#1}} \Dtcoh{#2}}

\renewcommand{\contentsname}{Contents\\{\footnotesize\normalfont(A table
of contents should normally not be included)}}

\usepackage{collectbox}

\makeatletter
\newcommand{\mybox}{%
    \collectbox{%
        \setlength{\fboxsep}{1pt}%
        \fbox{\BOXCONTENT}%
    }%
}
\makeatother

\begin{document}
\setcounter{tocdepth}{1}

\title{Perverse-Hodge complexes for Lagrangian fibrations and symplectic resolutions}
\author{Zhengze Xin}
\email{zhengze.xin@stonybrook.edu}
\address{DEPARTMENT OF MATHEMATICS, STONYBROOK UNIVERSITY, STONYBROOK, NY 11794, USA}
\begin{abstract}
We study perverse-Hodge complexes for Lagrangian fibrations on holomorphic symplectic varieties. We prove the symplectic Hard Lefschetz type theorem and the symmetry of perverse-Hodge complexes when the symplectic variety admits symplectic resolutions, therefore generalize the previous result by Schnell in the smooth case verifying a conjecture by Shen-Yin. Along the way, we study the perverse coherent properties of the intersection complex Hodge modules on symplectic varieties. As an application, we obtain an alternative proof of the numerical "perverse=Hodge" result by Felisetti-Shen-Yin, without using the Beauville-Bogomolov-Fujiki form. We also apply our results to study singular Higgs moduli spaces over reduced curves using results by Mauri-Migliorini on the local structure.
\end{abstract}

\maketitle
\tableofcontents

\section{Introduction}
\label{sec.introduction}
\subsection*{1.1 Holomorphic symplectic varieties} 
We work over complex number $\CC$. The goal of this paper is to study the Hodge theory of Lagrangian fibrations on holomorphic symplectic varieties, which are possibly singular and non-compact. The notion of holomorphic symplectic variety generalizes irreducible holomorphic symplectic manifolds, which are one of the building blocks for compact k\"ahler manifolds with numerically trivial canonical bundles by the Beauville-Bogomolov decomposition. In the sense of Beauville \cite{beauville1999symplectic}, a holomorphic symplectic variety is a normal variety $X$ which admits a non-degenerate closed holomorphic two-form on its regular locus $X_{reg}$ which extends holomorphically on some resolution of singularities of $X$. Classical objects such as Hilbert schemes of points and Nakajima quiver varieties are known to be symplectic. In terms of moduli problems, some important examples of symplectic varieties come from moduli spaces of sheaves such as Beauville-Mukai system or the moduli of Higgs bundles. 

The main objects to study in this article is a family of complexes of coherent sheaves called the \textit{perverse-Hodge complexes}, associated with a Lagrangian of a holomorphic symplectic variety. They are constructed using Saito's theory of mixed Hodge modules \cite{saito1988modules,saito1990mixed}. We will focus on their interaction with \textit{symplectic singularities}.

\subsection*{1.2 Perverse-Hodge complexes}
Let $f:M\rightarrow B$ be a Lagrangian fibration on holomorphic symplectic variety $M$ of dimension $2n$. 

When $M$ is compact and smooth, \textit{perverse-Hodge complexes} are originally defined in \cite{shen2023perverse-Hodgecomplex} by considering Saito's decomposition theorem for $Rf_*\QQ_M[2n]$. When $M$ is possibly singular and non-compact, the natural replacement for $\QQ_M[2n]$ is the intersection cohomology Hodge module $\ic_M$. 

Applying Saito's decomposition theorem \cite[\S 5.3]{saito1988modules} to the Lagrangian fibration $f$, we get the following decomposition
$$Rf_*\ic_M(n)=\bigoplus_{i=-n}^n \cP_i[-i]$$
where $\ic_M$ is the intersection cohomology Hodge module and $(\bullet)$ is the Tate twist. 
We can define the \textit{perverse-Hodge complex} as $G_{i,k}=\gr^F_{-k}\dr(\cP_i)[-i]$. In \cite{shen2023perverse-Hodgecomplex}, Shen and Yin proposed a symmetry between perverse-Hodge complexes. When $M$ is smooth, the conjecture was proven by Schnell \cite{schnell2023hodge} and we have symmetry of perverse-Hodge complexes:
    $$G_{i,k}\cong G_{k,i}.$$
in the derived category $D^b_{coh}(B)$.

We aim to investigate the perverse-Hodge complexes and the symmetry when $M$ is a possibly singular holomorphic symplectic variety. Our main theorems provides a partial answer to this question.
\begin{thm}[Symplectic Hard Lefschetz, Theorem \ref{semiperversesymphardlef}]\label{introsymmetry}
    If the intersection cohomology Hodge module $\ic_M$ is strongly coherent m-perverse (see Definition \ref{weaklyperverse}), i.e. it satisfies the following support condition on its graded de Rham:
    \begin{equation}\label{thm1.1support}
        \cd \supp \cH^j\gr^F_{-k}\dr(\ic_M)[k-2n]\ge 2j+2
    \end{equation}
    for $j\ge 1$, then the action by the reflexive symplectic form $\sigma\in H^0(M,\Omega^{[2]}_M)$ satisfies symplectic Hard Lefschetz theorem for $\ic_M$, i.e
    \begin{equation}
        \sigma^k:\gr^F_k\dr(\ic_M(n))[-k-n]\rightarrow \gr^F_{-k}\dr(\ic_M(n))[k-n].
    \end{equation}
    is an isomorphism.
\end{thm}

This provides a way to establish Hard Lefschetz type symmetry on symplectic varieties, by checking the dimension of the support of graded de Rham complexes of $\ic_X$. The advantage of this criteria is that it's an \'etale-local condition. 

The symmetry of perverse-Hodge complexes in the smooth case was established in \cite{schnell2023hodge}. Our second theorem generalizes it to the case when $M$ admits a symplectic resolution. 

\begin{thm}[Perverse-Hodge symmetry, Theorem \ref{sympP=H}]
    If $M$ admist a symplectic resolution, then the support condition in the above theorem holds and thus the symplectic Hard Lefschetz theorem holds. Moreover, we have the symmetry of perverse-Hodge complexes
    \begin{equation}
        G_{i,k}\cong G_{k,i}
    \end{equation}
    in the derived category $D^b_{coh}(B)$.
\end{thm}

In the compact case, the cohomology groups are finite dimensional. As shown in \cite{shen2022topology,shen2023perverse-Hodgecomplex}, the symmetry of perverse-Hodge complexes is in fact a sheaf-theoretic refinement of a numerical symmetry result concerning perverse numbers and Hodge numbers of $M$. In the singular setting we can define perverse Hodge numbers on the intersection cohomology $IH^*(M,\QQ)$, and we can reprove the following numerical result already observed by Felisetti-Shen-Yin \cite{felisetti2022intersection}.

\begin{thm}[Numerical "Perverse=Hodge", Corollary \ref{numP=H}]\label{intronumP=H}
    Let $M$ be a irreducible holomorphic symplectic variety which admits a symplectic resolution. We have
    \begin{equation}
        \pnum^{i,j}(f)=Ih^{i,j}(M).
    \end{equation}
\end{thm}

In \cite{felisetti2022intersection}, the identity between perverse numbers and Hodge numbers were established using the hyperk\"ahler metric and Looijenga-Lunts-Verbitsky (LLV) algebra, which relies essentially on the existence of the Beauville-Bogomolov-Fujiki (BBF) form. The symmetry for perverse-Hodge complexes allows one to deduce and refine the numerical result above without using the BBF form, which was considered to be one of the reason for introducing perverse-Hodge complexes as a categorification.

\subsection*{1.3 The support condition and the symplectic resolution}
The perverse coherent t-structures, in the sense of Deligne, Arinkin-Bezrukavnikov \cite{arinkin2009perverse} and Kashiwara \cite{kashiwara302086t}, is defined on the derived category of coherent sheaves as a commutative algebraic counter-part of the perverse t-structure on constructible sheaves. The notion is useful for our purpose in the following sense: being strongly m-perverse coherent as in Thm \ref{introsymmetry} allows us to extend the action of the symplectic forms in a unique way.

In the singular setting, the graded de Rham complexes $\gr^F_{-*}\dr(\ic_M)$ is the natural replacement for $\Omega_M^*$, it coincides with $\Omega_{M_{reg}}^*$ when restrict to the regular locus $M_{reg}$ of $M$. On the regular locus, the non-degenerate symplectic form $\sigma_{reg}$ acts on $\Omega_M^*$. 

we can verify the strongly m-perverse coherent condition when $M$ admits symplectic resolutions and thus extend the action of symplectic forms on $\Omega_M^*$ to the graded de Rham complexes.

\begin{thm}[Theorem \ref{5.4}]\label{1.4}
    Let $M$ be a holomorphic symplectic variety of dimension $2n$ admitting \'etale-locally symplectic resolutions. Then the graded de Rham complexes of $\ic_X$ are m-perverse coherent sheaves. Moreover, $\ic_X$ is strongly coherent m-perverse, satisfying
    $$
        \cd \supp \cH^j(\gr^F_{-k}\dr(\ic_X)[k-2n]) \ge 2j+2
    $$
    for $j\ge 1$. And therefore $M$ satisfies symplectic Hard Lefschetz symmetry Theorem \ref{introsymmetry}.
\end{thm}

The proof relies on the local structure of the holomorphic symplectic variety and the geometry of the symplectic resolution. The assumption on the existence of the symplectic resolution is necessary, and it's easy to construct counter-examples for which $\gr^F_{-k}\dr(\ic_X)$ are not m-perverse coherent. 

We emphasize that the support condition is an \'etale-local condition by the analytic nature of mixed Hodge modules. There are holomorphic symplectic varieties which only admit symplectic resolutions \'etale-locally but not globally, see \cite[Example 3.2(4)]{bakker2022global}. The theorem allows one to deduce the symplectic Hard Lefschetz symmetry for all singular Higgs moduli spaces over reduced curves, thanks to the local-model description using hypertoric varieties constructed by Mauri-Migliorini \cite{mauri2022hodge}, as we shall se below.

\subsection*{1.4 Moduli of Higgs bundles}

Let $C$ be a nonsingular irreducible projective curve of genus $g\ge 2$. The \textit{Dolbeault/Higgs moduli space} $M(n,d)$ is the coarse moduli space of rank $n$ degree $d$ semistable Higgs bundles on $C$. When $n, d$ are not coprime, the moduli space $M(n,d)$ is a non-proper singular symplectic variety with the Hitchin fibration $\chi(n,d):M(n,d)\rightarrow A_n$. The singular locus is given by strictly polystable Higgs bundles, which is related to partition of $n$ by multirank $\underline{n}$ and multiplicity $\underline{m}$. 

The \'etale-local structure of this moduli space is well-described using the deformation theory of Higgs bundles (cf. \cite{mauri2022hodge}). We can consider the open locus $A^{red}_n\subset A_n$ consisting of the reduced characteristic polynomials where multiplicities in $\underline{m}$ are $1$, and $M^{red}(n,d):=\chi(n,d)^{-1}(A^{red}_n)$. The moduli space over reduced Hitchin base $M^{red}(n,d)$ admits a Whitney stratification
$$M^{red}(n,d)=\coprod_{\underline{n}}M^\circ_{\underline{n}}(d)$$
and analytic normal slice $W_{\underline{n}}$ through $M^\circ_{\underline{n}}(d)$. Remarkably, $W_{\underline{n}}$ is isomorphic to a Nakajima quiver variety and admits symplectic resolutions. One can then verify the support condition for singular Higgs moduli spaces over the reduced Hitchin base. This gives the following theorem,

\begin{thm}[Theorem \ref{perv=Hodgehiggs}]\label{introsingularhiggs}
    Let $\chi(n,d):M^{red}(n,d)\rightarrow A_n^{red}$ be the Hitchin fibration over the reduced Hitchin base. Then the symplectic Hard Lefschetz theorem \ref{introsymmetry} holds for $M^{red}(n,d)$. If moreover $M^{red}(n,d)$ admits a symplectic resolution, then the perverse-Hodge complexes $G_{i,k}$ for $\chi(n,d)$ satisfy $G_{i,k}\simeq G_{k,i}$ in the derived category.
\end{thm}

\subsection*{1.5 Connection to other works}
We briefly review the history and other works along this direction. The symmetry theorem of perverse-Hodge complexes was originally inspired by the P=W conjecture and its compact analogy. The P=W conjecture by de Cataldo-Hausel-Migliorini \cite{de2012topologyP=W} predicts that the perverse filtration on $M(n,d)$ induced by the Hitchin morphism matches the weight filtration on the Betti moduli space $M_B(n,d)$ via non-abelian Hodge theory. When $n,d$ are coprime, the Higgs moduli space is non-singular and the conjecture has been recently resolved in \cite{maulik2022p=W1st,hausel2022p=W2nd,maulik2023perverse}. However, its singular variant remains open. 

The conjecture suggests that the perverse filtration of $M(n,d)$ can be computed via Hodge theory (up to a hyperk\"ahler twist), this motivates the numerical "Perverse=Hodge" for Lagrangian fibration $f:M\rightarrow B$ in the compact case (cf. \cite{shen2022topology}). The categorification of this symmetry phenomenon using Hodge modules leads to the notion of perverse-Hodge complexes $G_{i,k}$ (cf. \cite{shen2023perverse-Hodgecomplex}), which works in non-compact setting. 

Towards generalizing the results to the singular cases, the numerical "Perverse-Hodge" result and the LLV algebra structure on intersection cohomology were 
established in \cite{felisetti2022intersection} for irreducible symplectic varieties admitting symplectic resolutions, and \cite{tighe2022llv} when singularities are isolated. Our work provides partial answers on categorifying the results above on the level of perverse-Hodge complexes.

\subsection*{1.6 Organization}
The paper is structured as follows. Sections $\S 2$ and $\S 3$ consist of preliminaries on the theory of mixed Hodge modules which allows us to define perverse-Hodge complexes, and a brief review on the geometry of symplectic singularities. In section $\S 4$ we review subcategories of $D^b_{coh}$ defined by perverse coherent t-structures and discuss the extension property of complexes in them. We prove the main results in section $\S 5$, we first prove Theorem \ref{1.4} using the local product decomposition for symplectic singularities and the symplectic resolution around isolated singularities. By extending the action of symplectic forms on the regular locus, we establish the relative symplectic Hard Lefschetz theorem and then deduce Theorem \ref{introsymmetry}, as well as Theorem \ref{intronumP=H} in the compact situation. Finally, in section $\S 6$, we apply our result to singular Higgs moduli space and prove Theorem \ref{introsingularhiggs}.

\textbf{Notations and Set up}. Unless specified, we work over field of complex numbers $\CC$. A variety is an integral separated scheme of finite type over $\CC$. For a scheme $X$ and a point $x\in X$, the residue field at $x$ is denoted by $\kappa(x)$. We use the notion of holomorphic symplectic varieties in the sense of Beauville \cite{beauville1999symplectic}. In section $\S 5$, we denote $\gr^F_{-k}\dr(\ic_X)[k-2n]$ by $\cC_k$.

\subsection*{Acknowledgement}
The author would like to than Christian Schnell for suggesting the problem and consistent support, also for reading the early draft carefully and providing helpful feedback. He thanks Junliang Shen for giving a lecture series during a hyperk\"ahler workshop at the Simons Center in January of 2023, which initialized the author's interest on this topic. He thanks Mark Andrea de Cataldo for teaching him the foundation of Higgs moduli spaces. He would also like to thank Benjamin Tighe, David Fang, Qizheng Yin, Radu Laza, Siqing Zhang and Yoon-joo Kim for helpful discussions or comments.

\section{Preliminaries}
Consider a Lagrangian fibration $f:M\rightarrow B$. It is well-known that a nonsingular fiber is a complex torus. If $f$ is smooth, then $f$ is a family of abelian varieties. The variation of Hodge structures of $V=R^1f_*\QQ_M$ is the key to understanding the topology of the Lagrangian fibration.  
In general, in order to further control the topology of singular fibers, we need to consider the extensions of variation of Hodge structures across the locus where $f$ has singular fibers. Saito's theory of Hodge modules \cite{saito1988modules,saito1990mixed} provides a general framework to this end. 

\subsection{Overview of Hodge modules}
Saito introduced Hodge modules as vast generalizations of variation of Hodge structures, by allowing perverse sheaves instead of just local systems. Let $X$ be a complex manifold. The basic objects are filtered regular holonomic $\cD_X$-modules with $\QQ$-structure, given by a triple $M=(\cM, F_\bullet \cM, K)$ consisting of the following 
\begin{enumerate}
    \item $\cM$ is a regular holonomic right $\cD_X$-module with a good filtration $F_\bullet \cM$.
    \item $K$ is a perverse sheaf with an isomorphism 
    $$\dr(\cM)\simeq \CC\otimes_{\QQ}K$$
    via the Riemann-Hilbert correspondence.
\end{enumerate}

Hodge modules are such triples satisfying technical conditions which provides the stability when applying nearby and vanishing cycle functors. one can also define the notions of weight and polarization. Let $HM(X,w)$ (resp. $HM^p_Z(X,w)$) denote the category of (polarizable) Hodge modules of weight $w$ on $X$ (with strict support $Z$). For a possibly singular analytic space $X$, one can define Hodge modules on it by taking a global embedding into a complex manifold $i:X\hookleftarrow Y$ and taking all Hodge modules whose support is contained in the image $i(X)\subset Y$. 

In particular, for irreducible varieties $X$, there is a polarizable pure Hodge module $\ic_X^H$ associated to the intersection complex of $X$, which we call the intersection cohomology Hodge module.

Let $Z\subset X$ be an irreducible closed analytic subvariety. The main results in \cite{saito1990mixed} says that every polarizable variation of $\QQ$-Hodge structure of weight $w-\dim Z$ on a Zariski-open subset of $Z$ extends uniquely to an object of $HM^p_Z(X,w)$. Conversely, every object of $HM^p_Z(X,w)$ is obtained in this way. 
The category of mixed Hodge modules $MHM(X)$ can also be defined as the generalization of admissible variations of mixed Hodge structures.

The major accomplishments of Saito's theory is on generalizing the decomposition theorem of perverse sheaves developed in \cite{beilinson2018faisceaux}. It's a consequence of the following hard Lefschetz theoerem,

\begin{thm}[\cite{saito1988modules}, Theorem 5.3.1]\label{Saitorelhardlef}
Let $f:X\rightarrow Y$ be a proper morphism between two complex manifolds, and let $M\in HM^p(X,w)$. Let $\ell$ be the first Chern class of a $f$-ample line bundle on $X$. Then $\cH^if_*M\in HM^p(Y,w+i)$ and we have isomorphisms of Hodge modules \begin{equation}\label{relHL} \ell^i:  \cH^{-i}(f_*M)\xrightarrow{\simeq} \cH^{i}(f_*M)(i) \end{equation} 
where $(i)$ denotes the Tate twist shifting the Hodge filtration $F_{\bullet-i}$ by $i$ \end{thm}

\begin{cor}[Saito's Decomposition theorem]
    Let $f:X\rightarrow Y$ be a proper morphism and let $M$ be a pure Hodge module. Then there is a non-canonical isomorphism in the derived category of Hodge modules
    $$f_*M\cong \bigoplus \cH^i(f_*M)[-i].$$
\end{cor}

\subsection{Graded de Rham complexes}
To mixed Hodge modules one can associated with the graded de Rham complexes in the bounded derived category of coherent sheaves. Let $Y$ be a complex manifolds with pure dimension $d$ and $M=(\cM,F_\bullet \cM,K;W)$ be a mixed Hodge module on $Y$ with Hodge filtration $F_\bullet \cM$ and weight filtration $W$, where $\cM$ is the underlying regular holonomic right $\cD_Y$-module. 

Recall that for a regular holonomic right $\cD_Y$-module $\cM$, we have the de Rham functor
\begin{equation}
    \dr_Y(\cM):=\cM\otimes_{\cD_Y}^L \cO_X
\end{equation}
which is an object in the category of perverse sheaves $\Perv(Y;\CC_Y)$. Using the \textbf{Spencer resolution} (cf. \cite[Lemma 1.5.27]{hotta2007perverseDmodulereptheory}) this can be represented as a complex
\begin{equation*}
    \dr_Y(\cM)=\left[ M\otimes  \bigwedge^d\cT_Y\rightarrow \cdots \rightarrow M\otimes \cT_Y\rightarrow M \right]
\end{equation*}
which lives in cohomological degrees $-d,\dots, 0$. We call the complex above as the \textbf{de Rham complex} of the Hodge module $M$, and denote it by $\dr_Y(M)$. We sometimes abbreviate the notation as $\dr(M)$ when there's no confusion. 

\begin{defn}
    By taking the associated graded, the Hodge filtration on $\cM$ induces,
\begin{equation*}
    \gr^F_{-k}\dr(M)=\left[\gr^F_{-k-d}\cM\otimes \bigwedge^d \cT_Y \rightarrow\cdots \gr^F_{-k-1}\cM\otimes \cT_Y \rightarrow \gr^F_{-k}\cM \right]
\end{equation*}
which lives in cohomological degrees $-d,\dots, 0$. 
We call the above complex the \textbf{graded de Rham complex} of $M$. The construction above defines a functor:
\begin{equation*}
    \gr^F_{-k}\dr: D^bMHM(Y)\rightarrow D^b_{coh}(Y,\cO_Y)
\end{equation*}
from the bounded derived category of mixed Hodge modules to the bounded derived category of coherent sheaves.
\end{defn}

The construction of graded de Rham complex is compactible with proper pushforward and duality. 
\begin{lem}[\cite{kebekus2021extending}, Proposition 4.10]\label{fcommuDR}
    Let $f:X\rightarrow Y$ be a proper morphism, then we have the natural isomorphism of functors,
    \begin{equation*}
        Rf_*\circ \gr^F_{-k}\dr=\gr^F_{-k}\dr\circ Rf_*.
    \end{equation*}
\end{lem}

\begin{lem}[\cite{saito1988modules}, \S 2.4.3]\label{grdrdual}
    Let $\omega_X^\bullet$ be the dualizing complex of $X$ and $M$ be a Hodge module of weight $n$. Then for any integer $k$, we have an isomorphism in $D^b_{coh}(X,\cO_X)$
    $$R\cH om_{\cO_X}(\gr^F_{-k}\dr(M),\omega^\bullet_X)\cong \gr^F_{k}\dr(\mathbf{D}(M))$$
    where $\mathbf{D}$ is the dualizing functor for Hodge modules.
\end{lem}

The following lemma suggests that $\gr_*^F\dr(\ic_X)$ is a natural replacement for $\Omega^*_X$.

\begin{lem}\label{smoothgrdrQ}
    Let $X$ be a smooth complex variety of dimension $2n$, we have
    $$\gr^F_{-k}\dr(\QQ_{X}(n)[2n])[k-n]=\Omega_X^{n+k}[0].$$
\end{lem}
\begin{proof}
    $\QQ_X(n)[2n]$ underlies the (pure) Hodge module  of weight $0$, $$\QQ^H_X(n)[2n]=(\omega_X,F_\bullet\omega_X, \QQ_X(n)[2n])\in \HM(X,d_X).$$ The filtration (due to the Tate twist) is given by 
    \begin{equation*}
        F_{-n-1}\omega_X=0 \quad \text{and} \quad F_{-n}\omega_X=\omega_X.
    \end{equation*}
    Its associated graded de Rham complex is (in cohomological degree $-2n,\dots, 0$)
    \begin{equation*}
        \gr^F_{-k}\dr(\QQ_X(n)[2n])=\left[\gr^F_{-k-2n}\omega_X\otimes \bigwedge^{2n} \cT_X \rightarrow\cdots \gr^F_{-k-1}\omega_X\otimes \cT_X \rightarrow \gr^F_{-k}\omega_X \right]
    \end{equation*}
    which is just $\Omega_X^{n+k}[n-k]$.\end{proof}

\subsection{Perverse-Hodge complexes and symmetry}
Suppose $M$ is a holomorphic symplectic variety of dimension $2n$ which admits a Lagrangian fibration $f:M\rightarrow B$. We assume that $B$ is a complex manifold of dimension $n$. For irreducible hyperk\"ahler varieties, this implies $B$ is isomorphic to $\PP^n$ by Hwang's theorem(cf. \cite{hwang2008hwangthm}, \cite{bakker2023Hwangthm}). After taking the Tate twist by $n$, let $\operatorname{IC}_M(n)$ be the intersection complex of $M$ which underlies a (pure) Hodge module of weight $0$. By Saito's decomposition theorem, we have
$$Rf_*\ic_M(n)\cong\bigoplus_{i=-n}^n\cP_i[-i],$$
where $\cP_i$ are pure Hodge modules of weight $i$. More precisely, $\cP_i$ is the triple $$(\cP_i, F_\bullet \cP_i, \pervH^i(Rf_*\ic_M(n)[2n]))$$
where $\pervH^i(Rf_*\ic_M(n)[2n]))$ is the underlying perverse sheaf and $F_\bullet$ is the Hodge filtration on associated $\cD$-module.
\begin{defn}
    We define the perverse-Hodge complex
    $$G_{i,k}=\gr^F_{-k}\dr(\cP_i)[-i].$$
\end{defn}
\begin{rem}
    We adapt the indexing in \cite{schnell2023hodge} by taking a Tate twist $(n)$, which is different from the indexing in \cite{shen2023perverse-Hodgecomplex} where the perverse-Hodge complexes are originally introduced. One has $\cG_{i,k}=G_{i-n,k-n}$ in \cite{shen2023perverse-Hodgecomplex}. 
\end{rem}

Now we can state the main theorem on the symmetry of perverse-Hodge complexes,
\begin{thm}[Corollary \ref{sympP=H}]
    Let $f:M\rightarrow B$ be a Lagrangian fibration of holomorphic symplectic variety $M$ of dimension $2n$. If $M$ admits a symplectic resolution, then we have the symmetry of perverse-Hodge complexes
    \begin{equation*}
        G_{i,k}\cong G_{k,i}
    \end{equation*}
    in the derived category $D^b_{coh}(B)$.
\end{thm}

\section{Geometry of symplectic varieties}
In this section we review the notion of symplectic singularities and their geometry. We consider algebraic varieties over $\CC$.
\begin{defn}[\cite{beauville1999symplectic}]\label{defsympsing}
    A holomorphic symplectic singularity is a normal variety $X$ such that there is a symplectic form $\omega$ on the regular locus $X_{reg}$ of $X$, such that for some (equivalently, any) resolution of singularities $\pi:\tilde{X}\rightarrow X$, the pullback $\pi^*\omega$ to $\pi^{-1}(X_{reg})$ extends to a holomorphic $2$-form $\tilde{\omega}$ on $\tilde{X}$.

    If moreover the extended form $\tilde{\omega}$ is non-degenerate everywhere, then we say $\pi$ is a symplectic resolution.
\end{defn} 
A related notion generalizing symplectic singularities is a Poisson structure on schemes. We refer the readers to \cite{kaledin2006symplectic} for the definition and more details.

It is shown in \cite{beauville1999symplectic,elkik1981rationalite} that a holomorphic symplectic variety is rational Gorenstein, and has klt singularities. Differential forms on the regular locus of the symplectic variety can be extended to the resolution of singularities, as shown by Namikawa \cite{namikawa2001extension}. More generally, in \cite{kebekus2021extending}, the extension problem is related to the intersection cohomology Hodge module. In fact, let $j:X_{reg}\hookrightarrow X$ be the inclusion of the regular locus and let $\pi:\tilde{X}\rightarrow X$ be a resolution of singularities, assume $\dim(X)=2n$, then we have an isomorphism (cf. \cite[\S 8.4]{kebekus2021extending})
$$\pi_*\Omega^k_{\tilde{X}}\cong \cH^{-2n+k}(\gr^F_{-k}\dr(\ic_X)).$$
This identity will play an important role in the proof of the main theorems in section \S 5.

The following useful lemma tells us that the extension of differential forms on klt spaces coincides with reflexive differentials. 
\begin{lem}\label{extensionform}
    Every holomorphic form defined on $X_{reg}$ extends uniquely to a holomorphic form on $\tilde{X}$, and we have the isomorphism 
    \begin{equation*}
        \pi_*\Omega_{\tilde{X}}^k\cong j_*\Omega_{X_{reg}}^k=\Omega_{X}^{[k]}.
    \end{equation*}
\end{lem}
\begin{proof}
    Since $M$ has rational singularities, in particular it has weakly rational singularities, the first statement follows directly by \cite[Corollary 1.8]{kebekus2021extending}. By definition of rational singularities, we have $\omega_X^{GR}=\pi_*\omega_X=\omega_X$ and $\omega_X$ is reflexive. Thus we have isomorphism $\pi_*\omega_X=\omega_X\cong j_*\omega_{X_{reg}}$. By \cite[Theorem 1.4]{kebekus2021extending}, this implies 
    \begin{equation*}
        \pi_*\Omega_{\tilde{X}}^k\cong j_*\Omega_{X_{reg}}^k.
    \end{equation*}
    for all $0\le k\le 2n$. Finally, since $X$ is normal, $j_*\Omega_{X_{reg}}^k$ is the reflexive hull $\Omega_{X}^{[k]}$.
\end{proof}

A key ingredient for us is the local structure of the holomorphic symplectic variety $X$. We first recall the following stratification by the singular locus.

Canonical stratification for symplectic varieties\label{stratification}: By \cite[Proposition 3.1]{kaledin2006symplectic}, there is a canonical stratification $X_i\subset X$ by Poisson subschemes, such that the open parts $X^\circ_i$ are smooth and symplectic, by taking the singular locus successively. Namely, for $X_i\subset X$, we define $X_{i+1}$ to be the singular locus of $X_i$. For each closed point $x\in X_i^\circ$ in an open stratum, the local neighborhood of $x$ in $X$ admits a very nice product structure by the open stratum itself and a transversal slice which is a symplectic variety.

\begin{thm}[\cite{kaledin2006symplectic}, Theorem 2.3; \cite{kaplan2023productdecomp}, A.1]\label{weinsteindecomp}
    For any closed point $x\in X^\circ_i\subset X$, there is an analytic neighborhood $x\in U$ and a pointed Poisson isomorphism $(U,x)\cong (Y_x,y)\times (V,x)$ where $V$ is isomorphic to a neighborhood of $x$ in $X^\circ_i$, and $Y_x$ is a symplectic variety. We call each $Y_x$ the transversal slice in $X$ with respect to the canonical stratification.
\end{thm}

Kaledin's result in \cite{kaledin2006symplectic} is stated in terms of formal neighborhood, the proof in the appendix of \cite{kaplan2023productdecomp} allows us to replace formal neighborhood by analytic (\'etale) neighborhoods.

We have defined the notion of symplectic resolutions in Definition \ref{defsympsing}. For the purpose of induction in the proof of main theorems in \S 5, we show following that the existence of the symplectic resolution is compatible with the local structure.

\begin{prop}\label{propslicehassymplecticresolution}
    Suppose that $X$ is a symplectic variety with symplectic resolution $\pi:\tilde{X}\rightarrow X$. Let $x\in X^\circ_i\subset X$ be a closed point on the stratum $X^\circ_i$, suppose that $x$ has an \'etale neighborhood $(U,x)\cong (Y_x,y)\times (V,x)$ as in Theorem \ref{weinsteindecomp}. Then the symplectic variety $Y_x$ admits a symplectic resolution.
\end{prop}
\begin{proof}
    The restriction $\pi\vert_U: \tilde{U}:=\pi^{-1}(U)\rightarrow U$ is also a symplectic resolution. By generic smoothness, we can shrink $V$ so that the composition $p_2\circ \pi:\tilde{U}\rightarrow Y_x\times V\rightarrow V$ \cite[Corollary 10.7]{hartshorne2013algebraic} is smooth. In particular, we can find a closed point $v\in V$ so that $\tilde{Y_x}:=\pi^{-1}(Y_x\times \{v\})=(p_1\circ \pi)^{-1}(v)$ is smooth. This gives the following cartesian diagram (we use the symbols repetitively as there is no confusion)
    $$\begin{tikzcd}
\tilde{Y_x} \arrow[r, "\iota", hook] \arrow[d, "\pi"] & \tilde{U} \arrow[d, "\pi"] \\
Y_x\cong Y_x\times\{v\} \arrow[r, "\iota", hook]      & U\cong Y_x\times V     
\end{tikzcd}.$$
    We next show that $\pi:\tilde{Y_x}\rightarrow Y_x$ is a symplectic resolution. By \cite[Proposition 1.1]{fu2003symplectic}, a resolution is a symplectic resolution if and only if it's a crepant resolution. It suffices to show that $\pi^*(K_{Y_x})\cong K_{\tilde{Y_x}}$.
    Suppose that $Y_x\cong Y_x\times \{v\}\subset U$ is defined by the ideal sheaf $\cI_{Y_x}\subset \cO_U$, then $\tilde{Y}_x\subset \tilde{U}$ is defined by $\pi^*\cI_{Y_x}$. 
    Since $\pi:\tilde{U}\rightarrow U$ is a symplectic (crepant) resolution, we have $\pi^*(K_U)= K_{\tilde{U}}$. Thus $$\iota^*K_{\tilde{U}}=\iota^*\pi^*(K_U)=\pi^*\iota^*(K_U)=\pi^*K_{Y_x}$$
    
    Since $\tilde{Y_x}$ and $\tilde{U}$ are smooth, by adjunction we have
    $$K_{\tilde{Y_x}}=\iota^*K_{\tilde{U}}\otimes \operatorname{det}\cN_{\tilde{Y_x}/\tilde{U}}=\pi^*K_{Y_x}\otimes \operatorname{det}\cN_{\tilde{Y_x}/\tilde{U}}.$$
    Since $p_2\circ \pi:\tilde{U}\rightarrow V$ is smooth, by \cite[0473]{stacks-project} we have $\operatorname{det}\cN_{\tilde{Y_x}/\tilde{U}}=\pi^*\operatorname{det}(\cN_{\{v\}/V})\cong \cO_{\tilde{Y_x}}$.
    Thus $K_{\tilde{Y_x}}=\pi^*K_{Y_x}$, this shows that $\tilde{Y_x}$ is a symplectic (crepant) resolution for $Y_x$.
\end{proof}

\section{Perverse coherent t-structures}
We review the notion of perverse coherent sheaves. Perverse coherent sheaves can be viewed as the commutative algebraic counter-part of perverse sheaves from the topological perspective. The perverse coherent t-structure corresponds to the perverse t-structure on the derived category of constructible complexes via Riemann-Hilbert correspondence, and corresponds to the usual t-structure on the derived category of coherent sheaves via Grothendieck duality. We refer the readers to \cite{kashiwara302086t,arinkin2009perverse,bhatt2023applications} for more detailed discussions.

\subsection{Perverse coherent t-structures}
In this section we consider $X$ a finite type separated scheme defined over field $k$ which admits a dualizing complex. Here we choose $\omega^\bullet_X$ to be the normalized dualizing complex in the sense of \cite[\href{https://stacks.math.columbia.edu/tag/0A7M}{Tag OA7M}]{stacks-project}. 

\begin{defn}
    A perversity of $X$ is a function $p: X\rightarrow \ZZ$ on the underlying topological space of the scheme $X$, such that $p$ is non-decreasing under specialization, i.e, $p(y)\ge p(x)$ if $y\in \overline{\{x\}}$.
\end{defn}

Using the results in \cite{kashiwara302086t,arinkin2009perverse}, we can define two families of subcategories,
    \begin{equation*}
        \pervD^{\le k}_{coh}(X)=\{K\in D^b_{coh}(X)\mid i^*_xK\in D^{\le k+p(x)}_{coh}(\kappa(x)) \  \text{for all $x\in X$}\}
    \end{equation*}
    \begin{equation*}
        \pervD^{\ge k}_{coh}(X)=\{K\in D^b_{coh}(X)\mid i^{!}_xK\in D^{\ge k+p(x)}_{coh}(\kappa(x)) \  \text{for all $x\in X$}\}.
    \end{equation*}
It's shown that they form a bounded t-structure on $\Dbcoh(X)$ if and only if 
$$p(y)-p(x)\le \cd(y)-\cd(x)$$
for every $x,y\in X$ such that $y\in \overline{\{x\}}$. We write the heart as 
$$\Perv^p_{coh}(X)=\pervD^{\ge 0}_{coh}(X)\cap \pervD^{\le 0}_{coh}(X).$$

\begin{defn}[\cite{kebekus2021extending}, Lemma 18.4]
    Define perversity $m:X\rightarrow \ZZ$ with
    \begin{equation*}
    m(x)=\lfloor \frac{1}{2}\cd(x)\rfloor,
    \end{equation*}
    and its dual function
    \begin{equation*}
    \hat{m}(x)=\lceil \frac{1}{2}\cd(x)\rceil.
    \end{equation*}
    They define perverse t-structures
    \begin{equation*}
        \mervD^{\le k}_{coh}(X)=\{K\in D^b_{coh}(X)\mid \cd\supp \cH^i(K)\ge 2(i-k) \  \text{for all $i\in \ZZ$}\}
    \end{equation*}
    \begin{equation*}
        \mDtcoh{\hat{m}}{\geq k}(X)=\{K\in D^b_{coh}(X)\mid \cd\supp \cH^i(K)\ge 2(i-k)-1 \  \text{for all $i\in \ZZ$} \}
    \end{equation*}
    This also describes the subcategories with $\ge k$ by duality.
\end{defn}

For our convenience we introduce the following definition,
\begin{defn}\label{weaklyperverse}
    Let $X$ be a complex algebraic variety of dimension $2n$ which admits a dualizing complex. 
    \begin{enumerate}
        \item We say the intersection complex Hodge module $\ic_X$ is \textbf{coherent m-perverse}, if it satisfies the following support condition
    $$\cd\supp \cH^i(\gr^F_{-k}\dr(\ic_X)[k-2n])\ge 2i$$
    for all $i$ and $k$.
    \item We say the intersection complex Hodge module $\ic_X$ is \textbf{strongly coherent m-perverse}, if it satisfies the following support condition
    $$\cd\supp \cH^i(\gr^F_{-k}\dr(\ic_X)[k-2n])\ge 2(i+1)$$
    for all $i\ge 1$ and $k$.
    \end{enumerate}
\end{defn}

Let $\tau_{\le k}$ and $\tau_{\ge k}$ be the truncation functors with respect to the standard t-structure on $D^b_{coh}(X)$, and $$\mathbb{D}(-)=R\cH om(-,\omega^\bullet_X[-n]):D^b_{coh}(X)\rightarrow D^b_{coh}(X)^{op}$$ be the Grothendieck-duality functor, where $n$ is the dimension of $X$.

We also need the perverse function $p(x)=n-\dim\overline{\{x\}}$. Following from local duality, the next lemma shows that the $p$-perverse t-structure on $D^b_{coh}(X)$ is the Grothendieck-dual of the usual t-structure on $D^b_{coh}(X)$. 
\begin{lem}[\cite{arinkin2009perverse}, Lemma 3.3]\label{pervcohdual}
    Suppose $X$ is equidimensional and $K\in D^b_{coh}(X)$, then 
    \begin{enumerate}
        \item $K\in \pervD^{\le 0}(X)$ if and only if $\DD(K)\in D^{\ge 0}$.
        \item $K\in \pervD^{\ge 0}(X)$ if and only if $\DD(K)\in D^{\le 0}$.
    \end{enumerate}
    In particular, $K\in \Perv_{coh}^p(X)$ if and only if $\DD(K)$ is concentrated in degree $0$.
\end{lem}

\subsection{Intersection complexes of coherent sheaves}
\begin{defn}[\cite{schnell2015holonomicDmodabvar}, Definition 23.1]
    Let $\cF$ be a reflexive coherent sheaf. We call the  complex 
    \begin{equation*}
        \cIC(\cF)=(\tau_{\le\ell(X)-1}\circ \DD\circ \tau_{\le \ell(X)-2}\circ \cdots\circ \DD\circ \tau_{\le 1}\circ \DD)\cF
    \end{equation*}
    the \textbf{intersection complex} of $\cF$. Here $\ell(X)$ is the smallest odd integer such that $2\ell+1\ge \dim X$.
\end{defn}
\begin{defn}[\cite{schnell2015holonomicDmodabvar}, Definition 23.8]
    Let $j:U\hookrightarrow X$ be an open subset with $\cd(X\setminus U)\ge 2$, and $\cE$ be a locally free sheaf on $U$. Then we call the complex
    \begin{equation*}
        \cIC_X(\cE)=\cIC(j_*\cE)
    \end{equation*}
    the \textbf{intersection complex} of $\cE$ (with respect to $X$).
\end{defn}

\begin{prop}\label{reconstructionProp4.7}
    Let $\cC\in D^{\ge 0}_{coh}(X)$ be a bounded complex of coherent sheaves. Assume that there exists an open subset $j:U\hookrightarrow X$ with $\cd(X\setminus U)\ge 2$, such that $j^*\cH^0\cC$ is locally free. Suppose we have following support conditions
    $$\cd\supp\cH^j\cC\ge 2j+2\quad \text{and} \quad \cd\supp\cH^j\DD(\cC)\ge 2j+2$$
    for $j\ge 1$, then $\cC=\cIC_X(j^*\cH^0\cC)$, which is uniquely determined by the restriction $j^*\cH^0\cC$. 
\end{prop}
\begin{proof}
    We follow the same strategy as in \cite[Proposition 22.1]{schnell2015holonomicDmodabvar}. We show by induction that $\tau_{\le n}\cC$ can be reconstructed starting from $\cH^0\cC$ for each $n\ge 0$. The $n=0$ case is clear as $\tau_{\le 0}=\cH^0\cC$.

    For $j\ge n+1$, we have $\cd\supp\cH^j \cC\ge 2j+2\ge j+n+3$. Thus $\tau_{\ge n+1}\cC\in \pervD_{coh}^{\le -(n+3)}$. Consider the distinguished triangle,
    \begin{equation*}
        \tau_{\le n}\cC\rightarrow\cC\rightarrow\tau_{\ge n+1}\cC\xrightarrow{+1}\tau_{\le n}\cC[1]
    \end{equation*}
    
    Taking the Grothendieck dual, we see from Lemma \ref{pervcohdual} that $\DD(\tau_{\ge n+1}\cC)\in D^{\ge n+3}$. Therefore,
    \begin{equation*}
        \tau_{\le n+1}\DD(\cC)\simeq \tau_{\le n+1}\DD(\tau_{\le n}\cC).
    \end{equation*}
    
    Repeat the above argument but replace $\cC$ with $\DD(\cC)$, we have
    \begin{equation*}
    \tau_{\le n+2}\cC\simeq \tau_{\le n+2}\DD(\tau_{\le n+1}\DD(\cC))\simeq  \tau_{\le n+2}\DD(\tau_{\le n+1}\DD(\tau_{\le n}\cC)).
    \end{equation*}

    Repeat the above argument inductively until we reach $n=0$, in which case $\tau_{\le 0}\cC\simeq j_*j^*\cH^0\cC$ since $j^*\cH^0\cC$ is locally free in codimension $2$. Therefore we conclude that $\cC=\cIC_X(j^*\cH^0\cC)$.
\end{proof}

In particular, the proposition above applies when $\ic_X$ is strongly coherent m-perverse, as shown in the following important corollary.
\begin{cor}\label{reconstructionwhenicsemiperverse}
    Let $X$ be a complex algebraic variety which admits a dualizing complex. Suppose that $\ic_X$ is strongly coherent m-perverse, then 
    $$\gr_{-k}^F\dr(\ic_X(n))[k-n]=\cIC_X(\Omega^{n+k}_{X_{reg}}).$$
    In particular, the complex $\gr^F_{-k}\dr(\ic_X(n))$ is uniquely determined by its restrction to the smooth locus $j^*\cH^{k-n}\gr^F_{-k}\dr(\ic_X(n))=\Omega^{n+k}_{X_{reg}}$.
\end{cor}
\begin{proof}
    Since $\ic_X$ is a polarisable Hodge module, its de Rham complex is self-dual up to a shift in the filtration, see also \cite[Corollary 4.6]{kebekus2021extending}. Thus duality relates different graded quotients $\cC_k$. 
    More precisely, for $\cC_k:=\gr^F_{-k}\dr(\ic_X)[k-2n]$, the compatibility of $\gr\dr$ with duality gives $\DD(\cC_k)\cong \cC_{2n-k}[n]$ by Lemma \ref{grdrdual}. Since $\ic_X$ is strongly coherent m-perverse we have $\tau_{\ge 1}\cC_k\in \mervD_{coh}^{\le -1}(X)$, it's then easy to check that we also have $\tau_{\ge 1}\DD(\cC_k)\in \mervD_{coh}^{\le -1}(X)$. Now we can apply Proposition \ref{reconstructionProp4.7} and conclude. 
\end{proof}

\section{Proof of the main theorem}\label{mainthm}
We next prove our main theorem on the perverse coherent property of $\gr^F_{-k}\dr(\ic_X)$ when $X$ admits symplectic resolution. The result might be of independent interests as it encompasses the geometry of symplectic resolutions, while the statement is intrinsic. 

\subsection{Support conditions}
From now on we consider holomorphic symplectic variety $X$ of dimension $2n$. Throughout this section, we denote the complex $$\gr^F_{-k}\dr(\ic_X)[k-2n]\in D^{\ge 0}_{\coh}(X)$$ by $\cC_k$ to simplify the notation. Recall in Definition \ref{weaklyperverse} that when $X$ satisfies the \textbf{support condition (\ref{codimsupp}) for $\ic_X$}
\begin{equation*}\label{codimsupp}
        \cd \supp \cH^j\cC_k\ge 2j+2 \tag{$\star$}
\end{equation*}
for $j\ge 1$, we say that $\ic_X$ is \textbf{strongly coherent m-perverse}. 

If $X$ is smooth, then the support condition (\ref{codimsupp}) holds automatically because of Lemma \ref{smoothgrdrQ}. 

For $X$ singular with symplectic resolutions, we first prove a weaker lower bound using only the topological properties of symplectic resolutions. 
\begin{lem}
    Let $X$ be a holomorphic symplectic variety of dimension $2n$ which admits a symplectic resolution, and $\cC_k=\gr^F_{-k}\dr(\ic_X)[k-2n]$. Then 
    \begin{equation}\label{weakcodimsupp}
        \cd \supp \cH^j\cC_k\ge 2j
\end{equation}
for $j\ge 1$. Thus $\cC_k$ is a m-perverse coherent sheaf.
\end{lem}
\begin{proof}
    Let $\pi:\tilde{X}\rightarrow X$ be a symplectic resolution, in particular a semismall morphism by \cite[Lem 2.11]{kaledin2006symplectic}. By the decomposition theorem, $\ic_X$ is a direct summand of $R\pi_*\QQ_{\tilde{X}}[2n]$. Taking the graded de Rham and using its compatibility with proper pushforward, we have the following injection as a direct summand
    $$
        \cH^j\cC_k\hookrightarrow \cH^j(\gr^F_{-k}\dr(R\pi_*\QQ_{\tilde{X}}[2n])[k-2n])= R^j\pi_*(\gr^F_{-k}\dr(\QQ_{\tilde{X}}[2n])[k-2n])=R^j\pi_*\Omega^k_{\tilde{X}}
    $$
    Let $S_\ell \subset X$ be a stratum where the fiber $\pi^{-1}(x)$ has dimension $\ell$ for any $x\in S_\ell$. Since $\pi$ is semismall, $\cd S_\ell\ge 2\ell$.
    
    Consider the intersection $Z_\ell=S_\ell\cap \supp \cH^j\cC_k$. For $x\in S_\ell$, if $j>\ell=\dim \pi^{-1}(x)$, then $(R^j\pi_*\Omega^k_{\tilde{X}})_x=0$ by \cite[02V7]{stacks-project}. Therefore $Z_\ell$ is empty if $j>\ell$. Finally, we have
    $$\cd \supp\cH^j\cC_k=\cd \bigcup_{\ell\ge j}Z_\ell\ge \cd \bigcup_{\ell\ge j}S_\ell\ge 2j.$$
    \end{proof}

To show the stronger result that $\ic_X$ is strongly coherent m-perverse and establish the symplectic Hard Lefschetz symmetry, one has to investigate more on the geometry of the symplectic varieties. We will use an inductive argument using the local structure of symplectic varities and reduce to the isolated singularities case, in which situation we have enough vanishing results to conclude.

Let's first recall the product formula of mixed Hodge modules
\begin{lem}[\cite{maxim2011symmetric}, (1.4.1)]\label{boxprodgrdr}
    Let $M=(\cM, F_\bullet\cM, K_M)$, $N=(\cN,F_\bullet\cN, K_N)$ be Hodge modules on quasi-projective varieties $X$ and $Y$, $\cM\boxtimes \cN$ be the exterior product on $X\times Y$. We have
    \begin{equation*}
        \gr^F_{-k}\dr(\cM\boxtimes \cN)\cong \bigoplus_{i+j=k}\gr^F_{-i}\dr(\cM)\boxtimes \gr^F_{-j}\dr(\cN)
    \end{equation*}
\end{lem}

Applying the product formula, we obtain the following proposition which will be used for induction.
\begin{lem}\label{codimsuppinductionproduct}
   Let $X$ be a complex algebraic variety of dimension $2n$. Suppose that $X$ has a product decomposition $X=X_1\times \cdots \times X_k$ such that $\ic_{X_i}$ is strongly coherent m-perverse on each factor, then $\ic_X$ is also strongly coherent m-perverse.
\end{lem}
\begin{proof}
    For simplicity it suffices to prove the lemma when $k=2$. 
    
    On $X=X_1\times X_2$ we have $\ic_{X}\cong \ic_{X_1}\boxtimes \ic_{X_2}$. Let $d_{X_i}$ be the dimension of $X_i$. Let $j\ge 1$, we have the following
\begin{align*}\label{tensorinduction}
    &\cH^j(\gr^F_{-k}\dr(\ic_X)[k-2n])\\
    &\cong \bigoplus_{p+q=k}\cH^j(\gr^F_{-p}\dr(\ic_{X_1})[p-d_{X_1}]\boxtimes \gr^F_{-q}\dr(\ic_{X_2})[q-d_{X_2}])\\
    &\cong \bigoplus_{\alpha+\beta=j} \bigoplus_{p+q=k}\cH^\alpha(\gr^F_{-p}\dr(\ic_{X_1})[p-d_{X_1}])\boxtimes \cH^\beta(\gr^F_{-q}\dr(\ic_{X_2})[q-d_{X_2}])
\end{align*}
where the first isomorphism is given by the external product formula of Hodge modules, and the second equality is the K\"unneth formula in derived category.

Since on $X_i$, $\ic_{X_i}$ is strongly coherent m-perverse, thus 
$$
    \cd \supp \cH^\alpha(\gr^F_{-p}\dr(\ic_{X_i}[p-d_{X_i}])\ge 2\alpha+2
$$
for $\alpha\ge 1$ and $i=1$ or $2$. 

If $\alpha\ge 1$ and $\beta \ge 1$, we have 
$$
    \cd \supp \cH^\alpha(\gr^F_{-p}\dr(\ic_{X_1})[p-d_{X_1}])\boxtimes \cH^\beta(\gr^F_{-q}\dr(\ic_{X_2})[q-d_{X_2}])\ge 2(\alpha+\beta)+4>2j+2.
$$

If one of $\alpha,\beta$ is $0$, say $\alpha=0$, $\beta=j\ge 1$. In this case $\cH^0(\gr^F_{-p}\dr(\ic_{X_1})[p-d_{X_1}])$ may have support with dimension $d_{X_1}$(for example when $X_2$ is smooth). Thus 
$$
    \cd \supp \cH^0(\gr^F_{-p}\dr(\ic_{X_1})[p-d_{X_1}])\boxtimes \cH^j(\gr^F_{-q}\dr(\ic_{X_2})[q-d_{X_2}])\ge 2j+2.
$$
Therefore on $X$, $\ic_X$ is strongly coherent m-perverse.
\end{proof}

Now we show that when the symplectic singularity admits a symplectic resolution, we have the desired strong lower bound on the codimension of supports. 
\begin{thm}\label{5.4}
    Let $X$ be a holomorphic symplectic variety of dimension $2n$, which admits a symplectic resolution. Then $\ic_X$ is strongly coherent m-perverse, i.e. we have
    $$\cd\supp \cH^i(\gr^F_{-k}\dr(\ic_X)[k-2n])\ge 2(i+1)$$
    for all $i\ge 1$ and $k$.
\end{thm}
\begin{proof}
    We prove by induction on the dimension. 
    Let $\cS=\{\phi=X_{m+1}\subsetneq X_m\dots\subsetneq X_1\subsetneq X_0=X\}$ be the canoical stratification obtained by taking singular locus successively, i.e. $X_{i+1}$ is the singular locus of $X_{i}$. The open stratum $X^\circ_i:=X_i\setminus X_{i+1}$ is smooth and symplectic. 
    
    When $\dim X=2$, $X$ can only has isolated singularities as $X$ is normal. We will show later in Corollary \ref{suppisolated}, that if $X$ has a symplectic resolution and has only isolated singularities, then $\ic_X$ is strongly coherent m-perverse. In particular, $\ic_X$ is strongly coherent m-perverse when $X$ has dimension $2$. This settles the induction base. 
    
    From now we assume that the singular locus $X_1=X_{sing}$ has positive dimension.
    
    Recall that $\cC_k=\gr^F_{-k}\dr(\ic_X)[k-2n]$. The restriction of $\ic_X$ on the smooth locus $X_{reg}=X\setminus X_1$ is just $\QQ_{X_{reg}}[2n]$, where the shifted graded de Rham complex is $\gr^F_{-k}\dr(\QQ_{X_{reg}}[2n])[k-2n]$ which is concentrated at degree $0$, thus $\supp\cH^j\cC_k\subset X_1$ for $j\ge 1$. 

    Since mixed Hodge modules (therefore graded de Rham complexes) are analytically-local objects, the support condition (\ref{codimsupp}) for the graded de Rham complex of pure Hodge module $\ic_X$ can be checked locally. 
    
    Consider a positive dimension stratum $X_i$. Let $x\in X^\circ_i$ be a closed point in the open stratum. By Theorem \ref{weinsteindecomp} and Proposition \ref{propslicehassymplecticresolution} there exists a symplectic variety $Y_x$ as the transversal slice at $x$, which admits symplectic resolutions. We can find an analytic neighborhood such that $U\simeq Y_x\times X^\circ_{i,x}$. Then we have $\ic_{Y_x\times X^\circ_{i,x}}\cong \ic_{Y_x}\boxtimes \ic_{X^\circ_{i,x}}$. Here $\ic_{X^\circ_{i,x}}=\QQ_{X^\circ_{i,x}}[d_{X_1}]$ by smoothness. As in Lemma \ref{codimsuppinductionproduct}, we have
\begin{align*}
    &\cH^j(\gr^F_{-k}\dr(\ic_{Y_x\times X^\circ_{i,x}})[k-2n])\\
    &\cong \bigoplus_{p+q=k}\cH^j(\gr^F_{-p}\dr(\ic_{Y_x})[p-d_Y]\boxtimes \gr^F_{-q}\dr(\ic_{X^\circ_{i,x}})[q-d_{X_i}])\\
    &\cong \bigoplus_{\alpha+\beta=j} \bigoplus_{p+q=k}\cH^\alpha(\gr^F_{-p}\dr(\ic_{Y_x})[p-d_Y])\boxtimes \cH^\beta(\gr^F_{-q}\dr(\ic_{X^\circ_{i,x}})[q-d_{X_i}])\\
    &\cong \bigoplus_{\alpha+\beta=j} \bigoplus_{p+q=k}\cH^\alpha(\gr^F_{-p}\dr(\ic_{Y_x})[p-d_Y])\boxtimes \cH^\beta(\gr^F_{-q}\dr(\QQ_{X^\circ_{i,x}}[d_{X_1}])[q-d_{X_i}])\\
    &\cong \bigoplus_{\alpha+\beta=j} \bigoplus_{p+q=k}\cH^\alpha(\gr^F_{-p}\dr(\ic_{Y_x})[p-d_Y])\boxtimes \cH^\beta(\Omega^q_{X^\circ_{i,x}}[0])\\
    &\cong \bigoplus_{p+q=k}\cH^j(\gr^F_{-p}\dr(\ic_{Y_x})[p-d_Y])\boxtimes \Omega^q_{X^\circ_{i,x}}
\end{align*}

where the second isomorphism is given by the external product formula of Hodge modules, the third equality is the K\"unneth formula in derived category, the fourth equality comes from the fact that $X^\circ_{i,x}$ is regular, the fifth equality is Lemma \ref{smoothgrdrQ}.

Since $\dim Y_x<\dim X_x$, by induction hypothesis, we have 
$$\cd \supp \cH^j(\gr^F_{-p}\dr(\ic_{Y_x})[p-d_Y])\ge 2j+2 $$
for $j\ge 1$ on $Y_x$. Meanwhile, $\Omega^q_{X^\circ_{i,x}}$ have full support on $X^\circ_{i,x}$. Thus we also have 
$$\cd \supp \cH^j\cC_k|_U\ge 2j+2 $$
for $j\ge 1$ on the analytic \'etale neighborhood $U$ of $x\in X$. 

If the lowest stratum $X_m$ has positive dimension, run the above argument for all the (positive dimension) strata, we see the support condition (\ref{codimsupp}) holds on $X$ and $\ic_X$ is strongly coherent m-perverse. Otherwise, assume the lowest stratum $X_m$ has dimension $0$, which consists of isolated singularities in $X_{m-1}$ by the construction of the stratification. Since $X_m$ has codimension $2n$, it doesn't affect the support condition (\ref{codimsupp}) when $1\le j\le n-1$. When $j\ge n$, it suffices to show that $\cH^j\cC_k$ has no support, i.e. $\cH^j\cC_k=0$. This is proved in the lemma below.

In all, we have shown that $\ic_X$ is strongly coherent m-perverse if $\dim X=2n$. We conclude the proof by induction.
\end{proof}

\begin{lem}\label{singdim0}
    $\cH^j\cC_k=0$ for $j\ge n$
\end{lem}

Before we prove this lemma we recall vanishing theorems by Kaledin which will be used during the proof for reader's convenience, we refer to the original paper for complete proofs. 
\begin{lem}[\cite{kaledin2006symplectic}, Lemma 2.10]\label{Kaledinlemma2.10}
    Let $X$ and $Y$ be algebraic varieties over $k$. Assume that $X$ is smooth and let $\tau:X\rightarrow Y$ be a projective morphism. Then $R^p\tau_*\Omega_X^q=0$ whenever $p+q>\dim X\times_YX$. 
\end{lem}
The following lemma is from the proof of \cite[Lem 2.13]{kaledin2006symplectic}, 
\begin{lem}\label{Kaledinlemma2.13}
    Let $X$ be a symplectic variety and $\pi: \tilde{X}\rightarrow X$ be a projective birational map with smooth symplectic $X$. Let $E_x=\pi^{-1}(x)$ be the fiber over a closed point $x\in X$. Then $H^p(\mathfrak{X},\Omega_\mathfrak{X}^q)=0$ whenever $p>q$, where $\mathfrak{X}$ is the formal completion of $\tilde{X}$ along $E_x$. 
\end{lem}

\begin{proof}[Proof of the Lemma \ref{singdim0}]
    Consider the symplectic resolution $\pi: \tilde{X}\rightarrow X$. Let $X_{rel}\subset X_m$ be relevant strata in $X_m$, which in this case is just a discrete set of points $\{x_1, \dots, x_s\}$, associated to each point $\pi$ has dimension $n$ fiber $F_i=\pi^{-1}(x_i)$ for $1\le i\le s$. After truncation, the decomposition theorem for $\pi$ gives,
    \begin{equation}\label{eqlem5.2}
        R\pi_*\QQ_{\tilde{X}}[2n] \cong \ic_X\oplus \ic_{X_{rel}} (L_{X_{rel}})\oplus N\\
        \cong \ic_X\oplus(\bigoplus_{i=1}^s L_i) \oplus N
    \end{equation}
    where each $L_i$ is the skycraper sheaf on the point $x_i$ with stalk the same as cohomology of the fiber $H^{n}(F_i,\CC)$ as vector space, $N$ is a direct sum of perverse sheaves supported on positive dimensional relevant strata. As before, after taking $\cH^j$ and the graded de Rham complex to (\ref{eqlem5.2}) we have
    \begin{equation}\label{eq2lem5.2}
        R^j\pi_*\Omega^k_{\tilde{X}}\cong \cH^j\cC_k\oplus L_{j,k} \oplus \cH^j(\gr^F_{-k}\dr (N)[k-2n]),
    \end{equation}
    where $$L_{j,k}=\begin{cases}
        
    \bigoplus_{i=1}^s L_i \quad \text{if}\quad j+k=2n.
    \\  0 \quad \text{otherwise}.
    \end{cases}$$
    
    Let $x\in X_{sing}$ and take the formal completion at $x$, we have the following two cases:
    
    \textbf{Case 1}: If $x\in X_m\setminus X_{rel}$, then since the resolution $\pi$ is semismall, the fiber $\pi^{-1}(x)$ has dimension less than $n$. For dimension reason we have for $j\ge n$,
    $$\widehat{(R^j\pi_*\Omega_{\tilde{X}}^k)}_{x}=H^j(\fX,\Omega^k_{\fX})=0,$$
    where $\fX$ is the formal completion of $\tilde{X}$ along $\pi^{-1}(x)$. Thus $\cH^j\cC_k$ has no supports on $X_{sing}\setminus X_{rel}$.

    \textbf{Case 2}: Take $x=x_i\in X_{rel}$. Again for dimension reason we can assume $j=n$. We have $H^j(\fX_i,\Omega^k_{\fX_i})=0$ if $j=n>k$ according Lemma \ref{Kaledinlemma2.13}. If $k>n$, then $j+k=n+k>\dim(\tilde{X}\times_X\tilde{X})=2n$, thus $\widehat{(R^j\pi_*\Omega_{\tilde{X}}^k)}_{x_i}=0$ by Lemma \ref{Kaledinlemma2.10}. So we only need to deal with the case when $j=k=n$. Notice that the perverse sheaf $N$ is supported on positive dimensional relevant strata, over which the fibers of $\pi$ have dimension less than $n$, thus   $\cH^n(\gr^F_{-n}\dr(N)[-n])=0$. Take formal completion in equation (\ref{eq2lem5.2}), we have
    \begin{equation}\label{case2}H^n(\fX_i,\Omega^n_{\fX_i})=\widehat{(R^n\pi_*\Omega_{\tilde{X}}^n)}_{x_i}=\widehat{(\cH^n\cC_n)}_{x_i} \oplus L_i=\widehat{(\cH^n\cC_n)}_{x_i} \oplus H^{2n}(F_i,\CC).
    \end{equation}
    
    By the comparision of algebraic de Rham cohomology we have $H^{2n}(F_i,\CC)=H^{2n}_{\dr}(\fX)$. From the previous discussion we see when $p+q=2n$, $H^q(\fX,\Omega^p_{\fX})=0$ if $p\neq n$, $q\neq n$. 
    Now consider the spectral sequence of algebraic de Rham cohomology \cite[0FM6]{stacks-project}.
    \begin{equation}
        E_1^{p,q}=H^q(\fX,\Omega^p_{\fX}) \Rightarrow H^{p+q}_{DR}(\fX)
    \end{equation}
    with the following $E_1$-page:
    \begin{center}   
   \mybox{
    $$\begin{tikzcd}
             & 0                                                  & 0                                          &   \\
{} \arrow[r] & {H^n(\fX,\Omega^{n-1}_{\fX})=0} \arrow[r, "d_1"]   & {H^n(\fX,\Omega^n_{\fX})} \arrow[r, "d_1"] & 0 \\
{} \arrow[r] & {H^{n-1}(\fX,\Omega^{n-1}_{\fX})} \arrow[r, "d_1"] & {H^{n-1}(\fX,\Omega^n_{\fX})} \arrow[r]    & 0
\end{tikzcd}.$$} 
\end{center}

    Since for $p+q=2n$, there's only one non-trivial term $H^n(\fX,\Omega^n_{\fX})$, the associated spectral sequence degenerates and converges to $H^{2n}_{\dr}(\fX)=H^{2n}(F_i,\CC)$ at $E_1$-page, thus $H^n(\fX_i,\Omega^n_{\fX_i})=H^{2n}(F_i,\CC)$. Combined with (\ref{case2}) we see $\widehat{(\cH^n\cC_n)}_{x_i}=0=\cH^j\cC_k$. 

    In all, $\cH^j\cC_k$ vanishes and has no support for $j\ge n$.
\end{proof}

As a special case, when $X$ has only isolated singularities, we have the following corollary from the proof above,
\begin{cor}\label{suppisolated}
    Let $X$ be a holomorphic symplectic variety which admits symplectic resolutions. Suppose that $X$ has only isolated singularities, then $\ic_X$ is strongly coherent m-perverse.
\end{cor}

Finally, as we have shown that $\ic_X$ is strongly coherent m-perverse, we can reconstruct $\cC_k$ from its restriction to the regular locus. We can also generalize to the case when the symplectic resolution only exists \'etale-locally, by the analytic-local nature of Hodge modules from its construction.

\begin{cor}\label{reconfromXCor5.8}
    Let $X$ be a holomorphic symplectic variety which \'etale-locally admits symplectic resolutions. Then $\ic_X$ is strongly coherent m-perverse. And we have 
    $$\gr_{-k}^F\dr(\ic_X(n))[k-n]=\cIC_X(\Omega^{n+k}_{X_{reg}}).$$
    In particular, the complex $\gr^F_{-k}\dr(\ic_X(n))$ is uniquely determined by its restrction to the smooth locus $j^*\cH^{k-n}\gr^F_{-k}\dr(\ic_X(n))=\Omega^{n+k}_{X_{reg}}$.
\end{cor}
\begin{proof}
    In this case, $\ic_X$ is strongly coherent m-perverse by Corollary \ref{suppisolated}, thus we can apply Corollary \ref{reconstructionwhenicsemiperverse} and conclude since the Hodge module $\ic_X$ is \'etale-local.
\end{proof}

\begin{rem}
    We remark that the existence of local symplectic resolutions does not guarantee the existence of the global symplectic resolution. See \cite[Example 3.2 (4)]{bakker2022global}.
\end{rem}

Finally, we point out that we cannot expect $\ic_X$ to be m-perverse coherent if $X$ doesn't admit symplectic resolutions, even in the case when $X$ has only isolated singularities. 

\begin{eg}
To illustrate the failure of $\ic_X$ being m-perverse coherent, consider an isolated symplectic singularity $(X,x)$ of dimension $2n$ which is an affine cone of $Y$. Let $f:Z\rightarrow X$ be the resolution of the cone singularity $(X,x)$. As in the proof above, by taking the graded de Rham of the decomposition theorem, we obtain
$$Rf_*\Omega_{Z}^k\simeq \gr^F_{-k}\dr(\ic_X)\oplus \gr^F_{-k}\dr(N)$$
where $N$ a direct sum of perverse sheaves (up to shift) supported at the cone point $x$. 

If $\ic_X$ were m-perverse coherent, then $\cH^j\gr^F_{-k}\dr(\ic_X)=0$ when $j>n$. As shown in the proof of Lemma \ref{singdim0}, the stalk of $\cH^j\gr\dr(\ic_X)=0$ at $x$ is the difference between $H^*(f^{-1}(x),\CC)$ and the algebraic de Rham cohomology $H^*(\mathfrak{X},\Omega_{\mathfrak{X}}^k)$ of the formal scheme along $f^{-1}(x)$ . The non-degeneracy of the spectral squence and the non-vanishing of $H^*(\mathfrak{X},\Omega_{\mathfrak{X}}^k)$ obstruct the vanishing of $\cH^j\gr^F_{-k}\dr(\ic_X)$ for $j>n$ in general, contradicting the assumption that $\ic_X$ is m-perverse coherent.
\end{eg}

\subsection{Relative symplectic Hard Lefschetz theorem}

In this section we establish symplectic relative Hard Lefschetz theorem for Lagrangian fibrations of holomorphic symplectic varieties which admit symplectic resolutions, partially generalizing the main result in \cite{schnell2023hodge} to the singular case. 

Consider a Lagrangian fibration $f:M\rightarrow B$ of a holomorphic symplectic variety $M$. Suppose $M$ admits a symplectic resolution $\pi:\tilde{M}\rightarrow M$. The composition $g=f\circ \pi:\tilde{M}\rightarrow B$ is again a Lagrangian fibration.

From previous discussions we see that $\ic_M$ is strongly coherent m-perverse, namely
$$
        \cd \supp \cH^j\gr^F_{-k}\dr(\ic_M)[k-2n]\ge 2j+2
   $$
    for $j\ge 1$.

We have the decomposition theorem for $\pi, f$ and $g$,

\begin{equation}\label{decompi}
    R\pi_*\QQ_{\tilde{M}}(n)[2n]=\ic_M(n)\oplus N, \quad N=\bigoplus_{S_\alpha} \ic_{S_\alpha}(L_\alpha)(n)
\end{equation}
\begin{equation}\label{decomf}
    Rf_*\ic_M(n)=\bigoplus_{i=-n}^n\cP_{i}[-i].
\end{equation}
\begin{equation}
    Rf_*\ic_{S_\alpha}(L_\alpha)(n)=\bigoplus_{i=-n_\alpha}^{n_\alpha} \cP_{\alpha,i}[-i]
\end{equation}
\begin{equation}\label{decomg}
    Rg_*\QQ_{\tilde{M}}(n)[2n]=Rf_*\ic_M(n)\oplus Rf_*N=\bigoplus_{i=-n}^n\cP_{i}[-i]\oplus Rf_*N.
\end{equation}
\begin{equation}
    Rg_*\QQ_{\tilde{M}}(n)[2n]=\bigoplus_{i=-n}^n\tilde{\cP}_{i}[-i].
\end{equation}
where $L$ consists of pure Hodge modules with proper supports. We use sub-index $\alpha$ to denote those objects from the proper supports $S_\alpha$.

In particular we have,
\begin{align*}
    \tilde{\cP_i}=\cH^i(Rg_*\QQ_{\tilde{M}}(n)[2n])=\cH^i(\bigoplus_{i=-n}^n\cP_{i}[-i]\oplus Rf_*N)=\cP_i\oplus (\bigoplus_{\alpha} \cP_{\alpha,i})
\end{align*}
Recall that we have the perverse-Hodge complexes for $f$ 
$$
    G_{i,k}=\gr^F_{-k}\dr(\cP_i)[-i]
$$
$$
    G_{\alpha,i,k}=\gr^F_{-k}\dr(\cP_{\alpha,i})[-i]
$$
and respectively for $g$
$$
    \tilde{G}_{i,k}=\gr^F_{-k}\dr(\tilde{\cP}_i)[-i]=\gr^F_{-k}\dr(\cP_i)[-i]\oplus (\bigoplus_{\alpha} \gr^F_{-k}\dr(\cP_{\alpha,i})[-i])=G_{i,k}\oplus (\bigoplus_{\alpha} G_{\alpha,i,k})
$$

Let us discuss the action of symplectic forms on those complexes. The perverse-coherent property of $\gr^F_{-k}\dr(\ic_X(n))$ allows us to reconstruct it from the regular locus $M_{reg}$. We will show next that the extension of symplectic forms from $M_{reg}$ acts on the perverse-Hodge complexes. 

Now consider a closed, non-degenerate holomorphic symplectic form $\sigma_{reg}\in H^0(M_{reg},\Omega^2_{M_{reg}})$.

\begin{thm}[Symplectic Hard Lefschetz]\label{wholeisomsigma}
    Let $\sigma\in H^0(M,\Omega_M^{[2]})$ be the reflexive symplectic form which is an extension of a non-degenerate symplectic form $\sigma_{reg}$ on the regular locus. Suppose $M$ has symplectic resolutions \'etale-locally. Then it satisfies symplectic Hard Lefschetz theorem for $\ic_M$, i.e
    \begin{equation*}
        \sigma^k:\gr^F_k\dr(\ic_M(n))[-k-n]\rightarrow \gr^F_{-k}\dr(\ic_M(n))[k-n].
    \end{equation*}
    is an isomorphism. Furthermore, it induces an isomorphism in the derived category $D^b_{coh}(B)$
    \begin{equation*}
        \varphi_{\sigma^k}:\bigoplus_{i=-n}^{n-k}G_{i,-k}\rightarrow \bigoplus_{i=-n+k}^nG_{i,k}[2k]
    \end{equation*}
\end{thm}
\begin{proof}
    Since the restriction of $\sigma$ on $M_{reg}$ is a symplectic form on $M_{reg}$, the k-th wedge $\sigma^k_{reg}:\Omega^{n-k}_{M_{reg}}\rightarrow \Omega^{n+k}_{M_{reg}}$ induces an isomorphism for $1\le k\le n$. Applying $\mathcal{IC}_M$ and using Corollary \ref{reconfromXCor5.8}, it gives the following commutative diagram:
    $$\begin{tikzcd}
    \cIC_M(\Omega^{n-k}_{M_{reg}}) \arrow[r, "\simeq", no head] \arrow[d, "\sigma^k"] & {\gr^F_{k}\dr(\ic_M(n))[-k-n]} \arrow[d] \\
    \cIC_M(\Omega^{n+k}_{M_{reg}}) \arrow[r, "\simeq", no head]                       & {\gr^F_{-k}\dr(\ic_M(n))[k-n]}          
    \end{tikzcd}.                            
    $$
    where the vertical arrows are isomorphisms. Taking push-forward $Rf_*$ on the isomorphism $$\sigma^k:\gr^F_{k}\dr(\ic_M(n))[-k-n]\rightarrow\gr^F_{-k}\dr(\ic_M(n))[k-n],$$ and by the compatibility of the graded de Rham with proper push-forward, we get the isomorphism after the truncation
    \begin{equation*}
        \varphi_{\sigma^k}:\bigoplus_{i=-n}^{n-k}G_{i,-k}\rightarrow \bigoplus_{i=-n+k}^nG_{i,k}[2k]
    \end{equation*}
    where the range of the summation is controlled by \cite[\S44. Lemma]{schnell2023hodge}.
\end{proof}

Notice that we didn't use the existence of the symplectic resolution in the proof above. In fact we have
\begin{prop}\label{semiperversesymphardlef}
    Let $M$ be a holomorphic symplectic variety. Suppose that $\ic_X$ is strongly coherent m-perverse. Then the symplectic Hard Lefschetz theorem induced by the reflexive symplectic form holds. 
\end{prop}
\begin{proof}
    This simply follows from the proof of Theorem \ref{wholeisomsigma}. 
\end{proof}



Now we consider the cohomology class $[\tilde{\sigma}]\in H^2(\tilde{M},\CC)$ associated with the symplectic form $\tilde{\sigma}\in \Gamma(\tilde{M},\Omega^2_{\tilde{M}})$.

Since $\pi:\tilde{M}\rightarrow M$ is a symplectic resolution, we can find a non-degenerate symplectic form $\tilde{\sigma}\in \Gamma(\tilde{M},\Omega^2_{\tilde{M}})$ which is the extension of $\sigma_{reg}$ from $M_{reg}$, and satisfies $\tilde{\sigma}=\pi^*\sigma$. By the decomposition theorem, $\pi_*\tilde{\sigma}$ acts on $R\pi_*\QQ_{\tilde{M}}[2n]$. Pushing-forward along $f$ induces the following morphism
$$
    \tilde{\sigma}:\bigoplus_{i=-n}^{n}\tilde{G}_{i,-k}\rightarrow \bigoplus_{i=-n}^n\tilde{G}_{i,-k+2}[2] 
$$
and also the Deligne's decomposition for $\tilde{\sigma}$
$$
        \tilde{\sigma}_1: \tilde{G}_{i,-k}\rightarrow \tilde{G}_{i+1,-k+2}[2].
$$
Its k-th iterated morphism 
$$
        \tilde{\sigma}_1^k: \tilde{G}_{i,-k}\rightarrow \tilde{G}_{i+k,k}[2k] 
$$
induces
$$
    \mu_{i,-k}:G_{i,-k}\rightarrow G_{i+k,k}[2k].
$$

\begin{lem}\label{ictoN}
    Let $\pi:\tilde{X}\rightarrow X$ be a symplectic resolution of a symplectic variety $X$ of dimension $2n$. Consider the decomposition theorem for $\pi$:
    \begin{equation}
        R\pi_*\CC_{\tilde{X}}[2n]\cong \ic_X\oplus N.
    \end{equation}
    where $N$ is supported on the singular locus of $X$. Let $[\tilde\sigma]=[\pi^*\sigma]\in H^2(\tilde{X},\CC)$ be the cohomology class of the symplectic form on $\tilde{X}$. Then the composition morphism induced by $\tilde\sigma=\pi^*\sigma$
    \begin{equation}
        \pi_*\pi^*\sigma: \ic_X\hookrightarrow \ic_X\oplus N \xrightarrow{\pi_*\pi^*\sigma} \ic_X[2]\oplus N[2]\rightarrow N[2].
    \end{equation}
    vanishes.
    \end{lem}

\begin{proof} 
We follow the argument in \cite[Proposition 3.6]{de2022hitchin_fibration_ab_surf_P=W}. By extension theorem of differential forms Lemma \ref{extensionform}, we see that $\sigma=\pi_*\pi^*\sigma$.  By \cite[Remark 4.4.3]{de2005hodge}, the action of $\pi^*\sigma$ on $R\pi_*\CC_{\tilde{X}}[2n]\cong \ic_X\oplus N$ is a direct sum of actions by classes $\sigma_i=\sigma|_{S_i}$ on each summand $\ic_{\overline{S_i}}(L_i)$, where $S_i$ are relevant strata. In particular, the action $\pi_*\pi^*\sigma:\ic_X\rightarrow N[2]$ vanishes.
\end{proof}

Applying the above lemma to perverse-Hodge complexes yields the following corollary
\begin{cor}\label{compatible}
    We have $\tilde{\sigma}_1^k|_{G_{i,-k}}=\mu_{i,-k}$.
\end{cor}
\begin{proof}
    In Lemma \ref{ictoN}, we take $\tilde{X}$ (resp. $X$) to be $\tilde{M}$ (resp. $M$). Pushing-forward along the Lagrangian fibration $f$, the corollary follows immediately.
\end{proof}

Now we can deduce the relative symplectic Hard Lefschetz theorem in the case when $M$ has symplectic resolutions.
\begin{thm}[Relative symplectic Hard Lefschetz theorem]\label{symprelHLsection5}
    Let $M$ be a holomorphic symplectic variety admitting a symplectic resolution and $[\sigma]$ the reflexive symplectic form. Then the action of $[\sigma_M]$ induces an isomorphism 
    \begin{equation}\label{symprelHL}
       \mu_{i,-k} : G_{i,-k}\rightarrow G_{i+k,k}[2k].
    \end{equation}
    for every $k\ge 1$.
\end{thm}
\begin{proof}
     The morphism
    $$
        \tilde{\sigma}_1^k: G_{i,-k}\oplus (\bigoplus_{\alpha} G_{\alpha,i,-k})=\tilde{G}_{i,-k}\rightarrow \tilde{G}_{i+k,k}[2k]=G_{i+k,k}[2k]\oplus (\bigoplus_{\alpha} G_{\alpha,i+k,k}[2k]).
    $$
    is an isomorphism by the relative symplectic hard Lefschetz theorem (cf. \cite[\S15. Theorem]{schnell2023hodge}).
    By Lemma \ref{ictoN}, the action of $\tilde{\sigma}^k_1$ splits into a direct sum of the action on each summand. Thus the restriction of $\tilde{\sigma}^k_1$ on $G_{i,-k}$
    \begin{equation}
        \tilde{\sigma}^k_1|_{G_{i,-k}}=\mu_{i,-k}:G_{i,-k}\rightarrow G_{i+k,k}[2k]
    \end{equation}
    is an isomorphism.
\end{proof}

\subsection{Symmetry of perverse-Hodge complexes}

As a consequence of Hard Lefschetz theorem Theorem \ref{Saitorelhardlef}, by taking the graded de Rham complexes, the relative Hard Lefschetz theorem gives isomorphisms
\begin{equation}
    \ell^i: G_{-i,k}\rightarrow G_{i,i+k}[2i]
\end{equation}
for every $i\ge 1$ and $k\in \ZZ$.

\begin{thm}\label{sympP=H}
    Let $f:M\rightarrow B$ be a Lagrangian fibration of a holomorphic symplectic variety $M$ of dimension $2n$. If $M$ admits a symplectic resolution, then we have the following quasi-isomorphism of perverse-Hodge complexes in the derived category $D^b_{coh}(B)$,
    \begin{equation}
        G_{i,k}\cong G_{k,i}.
    \end{equation}
\end{thm}

\begin{proof}
    By Theorem \ref{symprelHLsection5} and the relative Hard Lefschetz theorem, the following composition is an isomorphism in the derived category $D^b_{coh}(B)$
    $$
        G_{i,k}\xrightarrow{\ell^{-i}}G_{-i,k-i}[-2i]\xrightarrow{\mu_{-i,k-i}} G_{-k,i-k}[-2k]\xrightarrow{\ell^k} G_{k,i}.
    $$
    Therefore we have $G_{i,k}\cong G_{k,i}$.
\end{proof}

\subsection{Numerical "Perverse=Hodge"}\label{numP=Hsection}
In the compact case, on the level of global cohomology, the symmetry of perverse-Hodge complexes gives rise to an equality of perverse numbers and the usual Hodge numbers. 

Let first recall the definition. For a possibly singular compact holomorphic symplectic variety with a Lagrangian fibration $f:M\rightarrow B$, the intersection cohomology $IH^*(M,\CC)$ carries a pure Hodge structure. We can consider the associated Hodge numbers
$$ Ih^{i,j}(M)=h^{i,j}(IH^*(M,\CC))\in \ZZ, \quad \text{for all $i,j\in \NN$}.$$

Meanwhile, the perverse t-structure associated with the Lagrangian fibration $f$ induces a perverse filtration on the intersection cohomology $IH^*(M,\CC)$
$$
    P_0IH^*(M,\CC)\subset P_1IH^*(M,\CC)\subset \cdots P_kIH^*(M,\CC)\subset \cdots IH^*(M,\CC).
$$

We can define the following invariants called the perverse numbers
$$
    \pnum^{i,j}(f)=\dim \operatorname{Gr}_i^P IH^{i+j}(M,\CC)=\dim(P_iIH^{i+j}(M,\CC)/P_{i-1}IH^{i+j}(M,\CC)).
$$

Using the symmetry of perverse-Hodge complexes, we now reprove the numerical result \cite{felisetti2022intersection}.
\begin{thm}[\cite{felisetti2022intersection}, Theorem 0.4]\label{numP=H}
    Let $f:M\rightarrow B$ be a Lagrangian fibration of compact holomorphic symplectic variety $M$ of dimension $2n$ which admits a symplectic resolution, then we have
    \begin{equation}
        \pnum^{i,j}(f)=Ih^{i,j}(M).
    \end{equation}
\end{thm}
\begin{proof}
    Applying the global cohomology to the symmetry, we have $H^{i+j}(B,G_{i,k})\simeq H^{i+j}(G_{k,i})$. Taking the direct sum with respect to the index $k$, we have
    $$
        \bigoplus_{k=-n}^n \dim H^{i+j}(B,G_{i,k})=\bigoplus_{k=-n}^n \dim H^{i+j}(B,\gr^F_{-k}\dr(\cP_i)[-i])=\dim H^{i+j}(B,\cP_i[-i]\otimes \CC).
    $$
    concerning perverse numbers. By the symmetry the direct sum above equals to 
    $$
        \bigoplus_{k=-n}^n \dim H^{i+j}(B,G_{k,i})=\bigoplus_{k=-n}^n \dim H^{i+j}(B,\gr^F_{-i}\dr(\cP_k[-k]))=\dim H^{i+j}(B,\gr^F_{-i}\dr(Rf_*\ic_M(n))
    $$
    concerning the Hodge numbers. This shows that
    $$
    \pnum^{i,j}(f)=Ih^{i,j}(M).
    $$
\end{proof}

\section{Singular Higgs moduli space}

In this section we apply our results to Hitchin fibrations of singular Higgs moduli space (without coprime condition). The results in this section may not be new and should be known to experts in the field.

The singular Higgs moduli space has symplectic singularities, and it admits symplectic resolution if and only if the base curve has genus 1, or the base curve has genus 2 and the rank $n=2$, see \cite{tirelli2019symplectic}. A remarkable feature of singular Higgs moduli spaces (over reduced Hitchin base) is that they are locally modeled on hypertoric quiver varieties \cite{hausel2002toric}. 

In the case of singular Higgs moduli space $M(n,d)$ over reduced curve with the Hitchin fibration $\chi(n,d)$. It is shown that locally it admits an embedding to the compactified Jacobian $\overline{J}_{d,\cB}$ of the spectral curves. Then we have the following result on the structure of the local model from the deformation theory of the Higgs bundles, we refer the readers to \cite{mauri2022hodge} for more details.

\begin{thm}[\cite{mauri2022hodge}, Theorem 7.4] (Local model of $\overline{J}_{d,\cB}$ and $M(n,d)$)\label{localmodelhiggs}
Let $\overline{J}_{d,\cB}$ be the compactified Jacobian. There exists an analytic neighborhood $V$ of an \'etale chart of $\overline{J}_{d,\cB}$ centered at $(C_a,\cI)$ such that the restriction to $V$ of the fibre product square as follows:
\begin{equation}
\begin{tikzcd}
{M(n,d)\supset U} \arrow[r, hook] \arrow[d, "{\chi(n,d)}", two heads] & {V\subset \overline{J}_{d,\cB}} \arrow[d, "\pi", two heads] \\
{\chi(n,d)(U)} \arrow[r, "f^\times", hook]                            & \pi(V)                                                     
\end{tikzcd}\end{equation}
with $U:=V\cap \chi(n,d)^{-1}(A_a)$, is locally isomorphic to 
\begin{equation}
    \begin{tikzcd}
    Y(\Gamma_{\underline{n}},0)\times \CC^a \arrow[r, "{(\tau_M,\ell_1)}", hook] \arrow[d, "{(\chi_{\Gamma_{\underline{n}}},\ell_2)}", two heads] & {X(\Gamma_{\underline{n}^\pm},0)\times \CC^b} \arrow[d, "{(\pi_{\Gamma_{\underline{n}}},\ell_4)}", two heads] \\
    \CC^{b_1(\Gamma_{\underline{n}})}\times \CC^c \arrow[r, "{(\tau_A,\ell_3)}", hook]                                                                    & \CC^s\times \CC^d                                                                                        
\end{tikzcd}
\end{equation}
where $Y(\Gamma_{\underline{n}},0)$ and $X(\Gamma_{\underline{n}^{\pm}},0)$ are hypertoric quiver varieties.
\end{thm}

\begin{lem}[\cite{mauri2022hodge}, Proposition 4.11]\label{sympresotorichiggs}
There exists a symplectic resolution of $Y(\Gamma_{\underline{n}},0)$ given by the affinization morphism,
\begin{equation}
    \pi_Y:Y(\Gamma_{\underline{n}},\theta)\rightarrow Y(\Gamma_{\underline{n}},0)
\end{equation}
where $Y(\Gamma_{\underline{n}},\theta)$ is a smooth holomorphic symplectic variety given by the quiver $\Gamma_{\underline{n}}$ and a generic stability vector $\theta$.
    
\end{lem}

Now combine our main theorem with the local structure of the Higgs moduli space, we deduce the following symmetry result on the singular Higgs moduli space. 

\begin{thm}
    Let $\chi(n,d):M(n,d)\rightarrow A_n^{red}$ be the Hitchin fibration over reduced Hitchin base. Then we have the following symplectic Hard Lefschetz theorem induced by the reflexive symplectic form on $M=M(n,d)$,
    \begin{equation}
        \sigma^k:\gr^F_k\dr(\ic_M(n))[-k-n]\rightarrow \gr^F_{-k}\dr(\ic_M(n))[k-n].
    \end{equation}
\end{thm}
\begin{proof}
    We first show that $M(n,d)$ satisfies the codimension support condition. 
    
    By Theorem \ref{localmodelhiggs} we see that $U\subset M(n,d)$ are locally isomorphic to $Y(\Gamma_{\underline{n}},0)\times \CC^a$. From Lemma \ref{sympresotorichiggs} we see that $Y(\Gamma_{\underline{n}},0)$ is a hypertoric quiver variety which admits symplectic resolution. Thus by Theorem \ref{5.4}, on both factors $Y(\Gamma_{\underline{n}},0)$ and $\CC^a$, the intersection complex Hodge modules are strongly coherent m-perverse, and so does $M(n,d)$ by Lemma \ref{codimsuppinductionproduct}. 
    
    Finally, the symplectic Hard Lefschetz theorem follows from Theorem \ref{wholeisomsigma}.
\end{proof}

\begin{thm}\label{perv=Hodgehiggs}
    If $M(n,d)$ admits a symplectic resolution, then the perverse-Hodge complexes $G_{i,k}$ for $\chi(n,d)$ satisfy symmetry $G_{i,k}\simeq G_{k,i}$ in the derived category.
\end{thm}
\begin{proof}
     This follows directly from Theorem \ref{sympP=H}.  
\end{proof}

\section{Further questions}
We conclude by listing some problems yet to be discovered.
    \begin{ques}
        Establish the symmetry $G_{i,k}\simeq G_{k,i}$ or its variants for singular symplectic varieties in general.
    \end{ques}
    
    \begin{ques}(Multiplicativity) 
    Are perverse-Hodge complexes multiplicative (via cup-product) with respect to both of the perverse index and the Hodge index?
    \end{ques}
    Such expectation comes from the multiplicativity of the perverse filtration in the proof of the P=W conjecture. On the level of global cohomology this was shown to be true in \cite[Appendix A]{shen2022topology}. As a folklore conjecture, the perverse filtration (on the sheaf-theoretic level) for the Lagrangian fibration is expected to have a multiplicative splitting . 

    \begin{ques}(Deformation invariance)
        Let $f:M\rightarrow B$ and $f':M\rightarrow B$ be two Lagrangian fibrations and $G_{i,k}$, $G'_{i,k}$ be perverse-Hodge complexes respectively. 
        Assume that $M$ and $M'$ are deformation invariant, do we have $G_{i,k}\simeq G'_{i,k}$?
    \end{ques}
    The numerical question has a positive answer. When $M$ and $M'$ are deformation equivalent irreducible symplectic varieties, the numercial perverse Hodge numbers are deformation invariant provided that $b_2\ge5$ in the compact case, see \cite[Theorem 0.1]{felisetti2022intersection}. 


\bibliographystyle{alpha}
\bibliography{wp_ref}
\end{document}